\documentclass[11pt]{amsart}
\usepackage{amscd,amssymb,float}

\newtheorem{theorem}{Theorem}[section]
\newtheorem{lemma}[theorem]{Lemma}

\renewcommand{\leq}{\leqslant}
\renewcommand{\geq}{\geqslant}

\theoremstyle{definition}

\theoremstyle{definition}

\numberwithin{equation}{section}

\numberwithin{equation}{section} \numberwithin{figure}{section}

\begin{document}
\title{Degenerations  of Ricci-flat  Calabi-Yau manifolds }
\author[X.Rong]{Xiaochun Rong * }
\thanks{*Supported partially  by NSF Grant DMS-1106517, and by research fund from Capital Normal
University.  \\ Address: Mathematics Department, Capital Normal
University, Beijing 100048, P.R.China, and  Mathematics Department,
Rutgers University New Brunswick, NJ 08903, USA.  E-mail address:
rong@math.rutgers.edu}
\author[Y. Zhang]{Yuguang Zhang **}
\thanks{**Supported by  NSFC-10901111.  \\ Address:   Mathematics Center and  Department,  Tsinghua 
University, Beijing 100084, P.R.China.   E-mail address:
 yuguangzhang76@yahoo.com}

\begin{abstract}
 This  paper is a sequel  to  \cite{RoZ}. We further investigate  the Gromov-Hausdorff convergence of  Ricci-flat K\"{a}hler metrics   under degenerations of Calabi-Yau manifolds.  We extend Theorem 1.1  in  \cite{RoZ} by removing the condition on existence of  crepant resolutions for Calabi-Yau varieties.

\end{abstract}

\maketitle
\section{Introduction}\label{sect0}\vskip 2mm

A Calabi-Yau manifold $M$ is a   projective
manifold with trivial canonical bundle $ \mathcal{K}_{M}\cong
\mathcal{O}_{M}$. Yau's theorem for  Calabi conjecture
(cf. \cite{Ya1})  asserts that, for
    any K\"{a}hler class $\alpha\in H^{1,1}(M, \mathbb{R})$,
there exists a
     unique Ricci-flat K\"{a}hler metric $g$ on $M$ whose  K\"{a}hler form
     $\omega$ represents $\alpha$. The metric behavior of Ricci-flat    Calabi-Yau manifolds is  studied by various authors from many perspectives  (cf.  \cite{An0}, \cite{CT}, \cite{GTZ},  \cite{GW2}, \cite{Kob1},  \cite{Lu1},  \cite{RoZ},  \cite{RZ}, \cite{Ti}, \cite{To},  \cite{To2},  \cite{Song}, \cite{Wi}). The present paper is a sequel  to  \cite{RoZ}, and we further investigate
   the Gromov-Hausdorff convergence of  Ricci-flat K\"{a}hler metrics   under degenerations of Calabi-Yau manifolds.

   The moduli space $\mathfrak{M}$ of a polarized Calabi-Yau manifold $(M, L)$, i.e. a Calabi-Yau manifold $M$ with an ample line bundle $L$, exists, and  is a quai-projective variety (cf.  \cite{Vie}).  Yau's theorem can be viewed as a map $$\mathcal{CY}: \mathfrak{M} \longrightarrow \mathcal{X} \ \ \ {\rm by} \ \ (M, L)\mapsto (M, g_{L}), $$ where $g_{L}$ is a Ricci-flat K\"{a}hler metric on $M$ whose K\"{a}hler class represents $c_{1}(L)$,  and $\mathcal{X}$ denotes the space of isometric classes of compact metric spaces with Gromov-Hausdorff topology.  An interesting  question is to understand the relationship between  the closure  of $\mathcal{CY}( \mathfrak{M})$ in $\mathcal{X}$ and the natural  algebro-geometric compactification  of $ \mathfrak{M}$.

 Assume that   $\mathcal{M}$ is an
$(n+1)$-dimensional variety, $\pi: \mathcal{M} \rightarrow \Delta$ is a proper flat morphism from $\mathcal{M}$ to a disc $\Delta\subset\mathbb{C}$,
$M_{0}=\pi^{-1}(0)$ is a singular projective variety,  $M_{t}=\pi^{-1}(t)$ is a smooth projective
$n$-dimensional manifold for any $t\in \Delta\backslash \{0\}$, and the relative canonical bundle $ \mathcal{K}_{ \mathcal{M}/\Delta}$ is trivial on $ \mathcal{M}\backslash M_{0}$. Then for any $t\in \Delta\backslash \{0\}$  $M_{t}=\pi^{-1}(t)$ is a Calabi-Yau
$n$-manifold, and $(\mathcal{M}, \pi)$ is called a {\it degeneration }  of Calabi-Yau manifolds over the disc $\Delta$. If $ \mathcal{L}$ is  an ample line bundle on  $ \mathcal{M}$ whose differential type is independent of $t$, then there is a map from $\Delta\backslash\{0\}$ to
the moduli space $\mathfrak{M}$ of $(M_{t}, \mathcal{L}|_{M_{t}})$,  and $0\in \Delta$ corresponds to a point in the algebro-geometric compactification    of  $\mathfrak{M}$.
A Calabi-Yau variety is a
  projective Gorenstein  normal  variety $M_{0}$ with  trivial canonical
sheaf $\mathcal{K}_{M_{0}}\cong \mathcal{O}_{M_{0}}$ and only
canonical singularities.

 In  this paper we   study  the Gromov-Hausdorff limit of  $\mathcal{CY}((M_{t}, \mathcal{L}|_{M_{t}}) )$ when $0\in \Delta$ represents a Calabi-Yau variety. The main result is:

\vskip 3mm
\begin{theorem}\label{t1.1}
 Let  $(\mathcal{M}, \pi)$ be a degeneration of Calabi-Yau manifolds over a disc $\Delta\subset \mathbb{C}$, and let $\mathcal{L}$ be an ample line bundle on $ \mathcal{M}$.
 Assume that   $M_{0}$ is a Calabi-Yau
$n$-variety ($n\geq 2$)   with singular set $S$, and    the
relative canonical bundle $\mathcal{K}_{\mathcal{M}/\Delta}$ of  $(\mathcal{M}, \pi)$ is trivial, i.e.,
$\mathcal{K}_{\mathcal{M}/\Delta}\cong \mathcal{O}_{\mathcal{M}}$.
 If  $\tilde{g}_{t}$ denotes   the unique Ricci-flat  K\"{a}hler  metric
with K\"{a}hler form $\tilde{\omega}_{t}\in
c_{1}(\mathcal{L})|_{M_{t}}\in H^{1,1}(M_{t}, \mathbb{R})$,  $t\in
\Delta\backslash \{0\}$, then   $$(M_{t}, \tilde{g}_{t})
\stackrel{d_{GH}}\longrightarrow (X, d_{X}),  $$ when $t\rightarrow 0$, in the Gromov-Hausdorff sense, where  $(X, d_{X}) $ denotes   the metric completion of
   $(M_{0}\backslash S,d_{g}) $,  $g$ is a Ricci-flat
  K\"{a}hler metric on $M_{0}\backslash S$, and $d_{g}$ is the
  Riemannian distance function of $g$.   Furthermore,   $X \backslash (M_{0}\backslash S)$  is a closed subset of
    Hausdorff dimension  less or equal to $2n-2$, and any tangent cone $T_{x}X$, $x\in X \backslash (M_{0}\backslash S)$, is not $\mathbb{R}^{2n}$.
\end{theorem}
\vskip 3mm

 In \cite{RoZ},  Theorem 1.1  is proved under an additional  assumption  that $M_{0}$ admits a crepant resolution, i.e. a resolution $\bar{\pi}: \bar{M}\rightarrow M_{0}$ with $\bar{\pi}^{*} \mathcal{K}_{M_{0}}=\mathcal{K}_{\bar{M}}$.  However, many Calabi-Yau varieties  do not admit any crepant resolution, e.g. Calabi-Yau varieties of dimension 4 with only finite ordinary double points as singularities.

In  \cite{RoZ}, it is proved under the hypothesis of Theorem    \ref{t1.1} that there is a uniform upper bound $D>0$ for the diameter  of $(M_{t},\tilde{g}_{t})$, i.e.   $$ {\rm diam}_{\tilde{g}_{t}}(M_{t})\leq D,  $$ and furthermore, $ F_{t}^{*}\tilde{g}_{t} $ converges to a Ricci-flat
  K\"{a}hler metric $g$
 in the local  $C^{\infty}$-sense on   $
M_{0}\backslash S$, when $t\rightarrow 0$,   where $F_{t}: M_{0}\backslash S\rightarrow M_{t}$ is a  smooth family of embeddings  with $F_{0}={\rm id}$. However,   very little is known about  the metric behavior  near singularities, and the global convergence.   The   Gromov pre-compactness theorem (cf.  \cite{G1}) asserts that, for any sequence $t_{k}\rightarrow 0$,  there is a subsequence $(M_{t_{k}}, \tilde{g}_{t_{k}})$ converging to a   compact length metric space $(X, d_{X})$ in the Gromov-Hausdorff sense. Moreover,  the structure of the limit metric space  $X$ is studied by Cheeger Colding and Tian (cf.  \cite{CC1} \cite{CC2} \cite{CT}), and it is shown that there is a closed subset $S_{X}\subset X$ of Hausdorff dimension less or equal to $2n-4$ such that any tangent cone $T_{x}X$, $x\in S_{X}$, is not $\mathbb{R}^{2n}$ and $X\backslash S_{X}$ is a smooth  manifold admitting a Ricci-flat K\"{a}hler metric $g_{\infty}$.  These results  will play a role  in the proof of Theorem 1.1.

 In Theorem  \ref{t1.1}, we obtain  that  $(X, d_{X})$ is isometric to the metric completion of  $(M_{0}\backslash S,d_{g})$, and the Gromov-Hausdorff convergence takes place for the whole continuous parameter $t\in
\Delta\backslash \{0\}$, i.e.  the Gromov-Hausdorff limit $X$ is unique and  thus  is  independent of the  choice of subsequences of $t$.  Furthermore, $(X\backslash S_{X}, g_{\infty})$ is isometric to $ (M_{0}\backslash S, g)$.

Finally, we remark that  after this paper is completed, we notice that Theorem 1.2 in  a  preprint   \cite{DS} posted  a few days ago  (see also  \cite{Ti1}) implies that  $X$ is homeomorphic to $M_{0}$.  We also mention that the homeomorphism  property    was  proved earlier   for  K3 surfaces in  \cite{Kob1}, and for  Calabi-Yau threefolds with only finite ordinary double points as singularities in \cite{Song} by assuming the existence of crepant resolutions.

\vskip 7mm

\noindent {\bf Acknowledgement:}    The second
 author  would like to  thank Professor  Jian Song  for helpful discussions, and Henri Guenancia for pointing out an error in the early version  of this paper.

 \vskip 10mm

\section{Proof of Theorem \ref{t1.1}}\vskip 2mm

Let  $(\mathcal{M}, \pi)$ be a degeneration of Calabi-Yau manifolds with trivial relative canonical bundle $\mathcal{K}_{\mathcal{M}/\Delta}$, and let $\mathcal{L}$ be an ample line bundle on $ \mathcal{M}$.
 Then  there is an embedding
$\mathcal{M}\hookrightarrow \mathbb{CP}^{N}\times \Delta$ such that
$ \mathcal{L}^{m}=\mathcal{O}_{\Delta}(1)|_{\mathcal{M}}$ for some
$m\geq 1$,  $\pi$ is a proper surjection given by
   the restriction of
 the projection from $\mathbb{CP}^{N}\times \Delta $ to $\Delta$, and
     the rank of $\pi_{*}$ is 1 on $
\mathcal{M}\backslash S$.
 Denote
$$\omega_{t}=\frac{1}{m}\omega_{FS}|_{M_{t}},$$
 where $\omega_{FS}$ is the standard
Fubini-Study metric on $\mathbb{CP}^{N}$, and  $g_{t}$ is  the
corresponding K\"{a}hler  metric of $ \omega_{t} $.  Let $ \Omega_{t}$ be a relative
holomorphic volume form, i.e.,     a  nowhere vanishing section of $
\mathcal{K}_{\mathcal{M}/\Delta}$.  Yau's theorem  of Calabi's
conjecture (\cite{Ya1})  asserts that    there is a unique
Ricci-flat K\"{a}hler  metric  $\tilde{g}_{t}$  with K\"{a}hler form
$\tilde{\omega}_{t}\in [\omega_{t}]= c_{1}(\mathcal{L})|_{M_{t}}\in
H^{1,1}(M_{t}, \mathbb{R})$ for $t\in \Delta\backslash \{0\}$, i.e.,
there is a   unique function $\varphi_{t}$ on $M_{t}$ satisfying
that $\tilde{\omega}_{t}=\omega_{t}+
\sqrt{-1}\partial\overline{\partial}\varphi_{t}$, and
\begin{equation}\label{e51.3}(\omega_{t}+
\sqrt{-1}\partial\overline{\partial}\varphi_{t})^{n}=(-1)^{\frac{n^{2}}{2}}e^{\sigma_{t}}\Omega_{t}\wedge
\overline{\Omega}_{ t},
 \ \ {\rm with} \ \ \sup_{M_{t}}\varphi_{t} =0
\end{equation} where    $\sigma_{t}$ is a constant depending only on $t$. By Lemma 3.1 in  \cite{RoZ}, \begin{equation}\label{eq01}
\|\varphi_{t}\|_{C^{0}}\leq \bar{C},
\end{equation} for a constant $\bar{C}$ independent of $t$, and
 by Lemma 3.2 in  \cite{RoZ},
\begin{equation}\label{eq1}
\omega_{t}\leq C \tilde{\omega}_{t},
\end{equation} for a constant $C>0$ independent of $t$.

Denote
 $f_{t}: (M_{t}, \tilde{\omega}_{t}) \longrightarrow (M_{t}, \omega_{t})\subset (\mathbb{CP}^{N},
\frac{1}{m}\omega_{FS})$  the inclusion   map induced by
$\mathcal{M}\subset \mathbb{CP}^{N}\times\Delta $. Note that
 $f_{t}$ is holomorphic, i.e. $\overline{\partial}f_{t}= 0$, $f_{t}^{*}\frac{1}{m}\omega_{FS}=\omega_{t} $  and   \begin{equation}\label{eq2} |d
f_{t}|^{2}_{\tilde{g}_{t},g_{FS}}=2 |\partial
f_{t}|^{2}_{\tilde{\omega}_{t},\omega_{FS}}=2{\rm
tr}_{\tilde{\omega}_{t}}f_{t}^{*}\omega_{FS}=2m{\rm
tr}_{\tilde{\omega}_{t}}\omega_{t}\leq C \end{equation} on $M_{t}$. Here the last inequality of
 (\ref{eq2})  follows from (\ref{eq1}).   So $f_{t}$ is a family of Lipschitz maps with a uniform
  Lipschitz constant independent of $t$.

By Theorem 1.4 in \cite{RoZ}, the diameter of  $(M_{t}, \tilde{g}_{t})$ is
uniformly bounded, i.e. $${\rm diam}_{\tilde{g}_{t}}(M_{t})\leq D,$$ for a constant $D$ indpendent of $t\in\Delta\backslash \{0\}$.  Furthermore,   for any smooth family of embeddings $F_{t}: M_{0}\backslash S\rightarrow M_{t}$ with $F_{0}={\rm id}$, $$ F_{t}^{*}\tilde{g}_{t} \longrightarrow  g, \ \ F_{t}^{*}\tilde{\omega}_{t} \longrightarrow  \omega, \  \ \  \varphi_{t}\circ F_{t} \longrightarrow  \varphi_{0} ,$$ $    {\rm
when} \ \
 t \rightarrow0,$
 in the $C^{\infty}$-sense on  any compact subset $K \subset
M_{0}\backslash S$,  where $\varphi_{0} $ is a smooth function,  $\omega$ is a Ricci-flat  K\"{a}hler form  on $M_{0}\backslash S$ obtained in \cite{EGZ}, and $g$ is   the corresponding K\"{a}hler metric  of $\omega$, i.e.  $\omega=\omega_{0}+\sqrt{-1}\partial \overline{\partial} \varphi_{0}$ and $$(\omega_{0}+
\sqrt{-1}\partial\overline{\partial}\varphi_{0})^{n}=(-1)^{\frac{n^{2}}{2}}e^{\sigma_{0}}\Omega_{0}\wedge
\overline{\Omega}_{ 0},$$ on $M_{0}\backslash S$.  By \cite{ZZ2},   $\varphi_{0} $ extends to a continuous function on $M_{0}$, still   denoted  by $\varphi_{0} $.

  \vskip 2mm

  \begin{proof}[Proof of Theorem \ref{t1.1}]
  By Gromov compactness theorem (cf.  \cite{G1}), there is a subsequence $(M_{t_{k}}, \tilde{g}_{t_{k}})$ converging to a   compact length metric space $(X, d_{X})$  in the Gromov-Hausdorff sense, i.e.  there  exists a $\epsilon (k) $-approximation  $\phi_{k}: M_{t_{k}}\rightarrow X$ where $\epsilon (k)\rightarrow 0$ when $k\rightarrow\infty$.

  Recall that the convergence for Riemannian manifolds with Ricci-curvature bounded from below was intensively studied by Cheeger and Colding (cf. \cite{CC1} \cite{CC2} etc.).
 Applying  Section 6 and 7 of \cite{CC1} to our circumstances,  we conclude that   there is a closed subset $S_{X}\subset X$ of Hausdorff
  dimension $\dim_{\mathcal{H}}S_{X} \leq 2n-2$ such  that, for any $x\in  S_{X}$, any  tangent cone $T_{x}X$ is not isometric to $\mathbb{R}^{2n}$.
   Furthermore, $X\backslash S_{X}$ is a smooth open  complex manifold,  and $d_{X}|_{X\backslash S_{X}}$ is induced
  by a Ricci-flat K\"{a}hler metric $g_{\infty}$ on $X\backslash S_{X}
  $.     By Section 7 of \cite{CC1},   $\tilde{g}_{t_{k}}$ smoothly converges to $g_{\infty}$ on $X\backslash S_{X}
  $ under the Gromov-Hausdorff convergence of $(M_{t_{k}}, \tilde{g}_{t_{k}})$ to $(X, d_{X})$, i.e.,  there are   compact subsets   $K_{1}\subset \cdots \subset K_{k} \subset K_{k+1}\subset \cdots \subset X\backslash S_{X}$ such that $\bigcup_{k=1}^{\infty} K_{k}= X\backslash S_{X}$, $\phi_{k}|_{\phi_{k}^{-1}(K_{k})}: \phi_{k}^{-1}(K_{k}) \rightarrow K_{k}$ can be chosen as  diffeomorphisms  and $$\phi_{k}^{-1}|_{K_{k}}^{*}\tilde{g}_{t_{k}}\longrightarrow g_{\infty}, \ \ \ \phi_{k,*}J_{t_{k}}\phi^{-1}_{k,*}\longrightarrow J_{\infty}$$ on any compact subset $K\subset\subset X\backslash S_{X}
  $ in the $C^{\infty}$-sense, where    $J_{t_{k}}$ (resp. $J_{\infty}$) denotes   the complex structure of $M_{t_{k}}$ (resp. $X\backslash S_{X}$).

  The same argument as in the proof of Lemma 4.1 in   \cite{RZ} implies that,
by passing to a subsequence,   $\phi_{k}\circ F_{t_{k}}: M_{0}\backslash S \rightarrow X$ converges to a local isometric embedding $\Psi: (M_{0}\backslash S, g)\rightarrow (X, d_{X})$, i.e. for a  $y\in M_{0}\backslash S$, there is a $\rho_{y}>0$ such that $d_{X}(\Psi(y_{1}), \Psi(y_{2}))=d_{g}(y_{1}, y_{2})$ for any $y_{1}, y_{2}\in B_{g}(y, \rho_{y})$.  Moreover, for any $y\in M_{0}\backslash S$, $F_{t_{k}}(y)\in M_{t_{k}}$ converges to a point $x\in X$ under the Gromov-Hausdorff convergence of $(M_{t_{k}}, \tilde{g}_{t_{k}})$ to $(X, d_{X})$, and $\Psi (y)=x$.

\vskip 2mm
\begin{lemma}\label{c1} A subsequence of    $f_{t_{k}}: M_{t_{k}} \rightarrow \mathbb{CP}^{N}$ converges to a continuous map  $f_{\infty}: X\rightarrow \mathbb{CP}^{N}$ which satisfies that $f_{\infty}( X)=M_{0}$,   $f_{\infty}\circ \Psi={\rm id}: M_{0}\backslash S \rightarrow M_{0}\backslash S $, and  $f_{\infty}|_{X\backslash S_{X}}$ is holomorphic. Furthermore,  $$ f_{t_{k}}\circ \phi_{k}^{-1}|_{K} \longrightarrow f_{\infty}|_{K},$$ in the $C^{\infty}$-sense on any compact subset $K\subset X\backslash S_{X}$.
\end{lemma}
\vskip 2mm

  \begin{proof}
By  (\ref{eq2}),   $f_{t_{k}}$ are  Lipschitz maps with a uniform
  Lipschitz constant independent of $k$. Hence,  passing to a subsequence,    $f_{t_{k}}: M_{t_{k}} \rightarrow \mathbb{CP}^{N}$ converges to a continuous map  $f_{\infty}: X\rightarrow \mathbb{CP}^{N}$ when $t_{k}\rightarrow 0$ (cf. \cite{Rong}), i.e. for any sequence $p_{k}\in M_{t_{k}}$ which converges to $x\in X$ under the Gromov-Hausdorff convergence of $(M_{t_{k}}, \tilde{g}_{t_{k}})$ to $(X, d_{X})$,  we have that  $f_{t_{k}}(p_{k})$ converges to $f_{\infty}(x)$ in $\mathbb{CP}^{n}$.   From that  $f_{t_{k}}(M_{t_{k}})=M_{t_{k}}\subset \mathbb{CP}^{N}$, we see that  $f_{\infty}( X)=M_{0}\subset \mathbb{CP}^{N}$.  Note that  for any $y\in M_{0}\backslash S$, $F_{t_{k}}(y)\in M_{t_{k}}$ converges to a point $x\in X$ under the Gromov-Hausdorff convergence of $(M_{t_{k}}, \tilde{g}_{t_{k}})$ to $(X, d_{X})$, and $F_{t_{k}}(y)$ converges to $y$ in $\mathbb{CP}^{N}$.  By the convergence of $f_{t_{k}}$, $f_{t_{k}}(F_{t_{k}}(y))=F_{t_{k}}(y)$ and $F_{0}={\rm id}$,  we have that  $f_{\infty}(x)=y$.  Thus
  $f_{\infty}\circ \Psi={\rm id}: M_{0}\backslash S \rightarrow M_{0}\backslash S $.

 We claim that $f_{\infty}|_{X\backslash S_{X}}$ is holomorphic.  Note that   $f_{t_{k}}: M_{t_{k}} \longrightarrow \mathbb{CP}^{N}$ is holomorphic, i.e. $$\overline{\partial}f_{t_{k}}=\frac{1}{2}(df_{t_{k}}+J\circ df_{t_{k}} \circ J_{t_{k}})= 0, $$ where $J$ denotes the complex structure of $\mathbb{CP}^{N}$.
   Then $f_{t_{k}}: (M_{t_{k}}, \tilde{g}_{t_{k}}) \longrightarrow (\mathbb{CP}^{N}, g_{FS})$ is a harmonic map, and  by (\ref{eq2})  $$|d
f_{t_{k}}|^{2}_{\tilde{g}_{t_{k}},g_{FS}}\leq C $$ for a constant $C>0$ independent of $k$.
   In local  coordinates on $M_{t_{k}}$ and $\mathbb{CP}^{N}$ we have $$\Delta_{\tilde{g}_{t_{k}}}f_{t_{k}}^{i}+\tilde{g}_{t_{k}}^{\alpha\beta}\Gamma^{i}_{jl}\frac{\partial f_{t_{k}}^{j}}{\partial x^{\alpha}}\frac{\partial f_{t_{k}}^{l}}{\partial x^{\beta}}=0,$$ where $\Gamma^{i}_{jl}$ denotes  the Christoffel symbols of $g_{FS}$ (cf. \cite{Sc}).  On any compact subset $K\subset\subset X\backslash S_{X}
  $,  $\|f_{t_{k}}^{i}\|_{C^{1,\alpha}}\leq C_{K}$ for a constant $C_{K}>0$ independent of $k$ by  the standard Schauder estimates (cf    \cite{GT}),
    since $\phi_{k}^{-1}|_{K_{k}}^{*}\tilde{g}_{t_{k}}$   converges to $ g_{\infty}$ in the $C^{\infty}$-sense, and
        $|\frac{\partial f_{t_{k}}^{j}}{\partial x^{\alpha}}|_{C^{0}}\leq C$.   Here we identify $K$ and $\phi_{k}^{-1}(K)$ for $k\gg 1$.  The  elliptic bootstrap   estimates show that  $\|f_{t_{k}}^{i}\|_{C^{l}}\leq C_{l,K}$ for any $l\in \mathbb{N}$ on $K$. Hence,   passing a subsequence,  $f_{t_{k}}$  converges to $f_{\infty}$ smoothly on $K$, and, by the convergence of $J_{t_{k}}$ to $J_{\infty}$,  $f_{\infty}$ is holomorphic, i.e. $$\overline{\partial}f_{\infty}=\frac{1}{2}(df_{\infty}+J\circ df_{\infty} \circ J_{\infty})= 0. $$ By the standard  diagonal argument, $f_{\infty}$ is holomorphic on $X\backslash S_{X}$.  \end{proof}

       For any $x_{0}\in \Psi(M_{0}\backslash S)$, the tangent cone $T_{x_{0}}X\cong \mathbb{R}^{2n}$ since $\Psi$ is a local isometry. Thus  $\Psi(M_{0}\backslash S)\subset X\backslash S_{X}$.     By the volume convergence theorem in \cite{CC2},    we have $$ \mathcal{H}^{2n}(X,d_{X})= {\rm Vol}_{\tilde{g}_{t_{k}}}(M_{t_{k}})=\frac{1}{n!}c_{1}(\mathcal{L}|_{M_{t_{k}}})^{n}=\frac{1}{n!}c_{1}(\mathcal{L}|_{M_{0}})^{n}= {\rm Vol}_{g}(M_{0}\backslash S), $$  where $\mathcal{H}^{2n}(X,d_{X})$ denotes the Hausdorff  measure  of $(X, d_{X})$, and thus $${\rm Vol}_{g_{\infty}}(X \backslash S_{X}) =  \mathcal{H}^{2n}(X,d_{X})= {\rm Vol}_{g_{\infty}}(\Psi(M_{0}\backslash S)),$$    which implies that $\Psi(M_{0}\backslash S)$ is dense in $(X,d_{X})$.  
       If we denote  $E=X\backslash (S_{X}\bigcup \Psi(M_{0}\backslash S))$, then we claim that  $f_{\infty}(E)\subset S$.  If there is a $x\in E$ such that $y=f_{\infty}(x)\in M_{0}\backslash S$, then  we take a sequence of points  $x_{i}\in \Psi(M_{0}\backslash S)$ converging  to $x$, i.e.  $x_{i}  \rightarrow  x$ in $X$.   We have a sequence $y_{i} \in M_{0}\backslash S$ such that  $\Psi (y_{i})=x_{i}$, and by $f_{\infty}\circ \Psi={\rm id}$,  $y_{i}=f_{\infty}(x_{i}) $ converges to $y$ in $M_{0}\backslash S$.   Hence $x_{i}$ converges to $\Psi (y) $ in $\Psi(M_{0}\backslash S)$, which is a contradiction.

           Since  $f_{\infty}|_{X\backslash S_{X}}$ is holomorphic, $E=f_{\infty}|_{X\backslash S_{X}}^{-1}(S)$ is a complex subvariety of $(X\backslash S_{X}, J_{\infty})$ with $\dim_{\mathbb{C}} E \leqslant n-1$, and $f_{\infty}|_{X\backslash (S_{X}\cup E)}: X\backslash (S_{X}\cup E)\rightarrow M_{0}\backslash S$ is bi-holomorphic.    Thus $E\bigcup S_{X}=X\backslash  \Psi(M_{0}\backslash S)$ has the  Hausdorff codimension bigger or equal to $2$.  
          By    Section 3 of  \cite{CC2}, we obtain  that  for any $x_{1}, x_{2}\in \Psi(M_{0}\backslash S)$
    and any $\delta >0$, there is  a  curve $\gamma_{\delta}$ connecting $x_{1}$ and $
 x_{2}$ in  $\Psi(M_{0}\backslash S)$ such that $$d_{X}(x_{1},x_{2}) \leq  {\rm length}_{d_{X}}(\gamma_{\delta}) \leq
  \delta
  +d_{X}(x_{1},x_{2}).$$  Thus $(X, d_{X})$ is isometric to the metric completion of   $(M_{0}\backslash S,d_{g}) $.   For any other convergence subsequence $(M_{t'_{k}}, \tilde{g}_{t'_{k}})$, the same argument shows that the limit is still the metric   completion of   $(M_{0}\backslash S,d_{g}) $.  We obtain that $(M_{t}, \tilde{g}_{t})$ converges to $(X, d_{X})$ in the Gromov-Hausdorff sense  when $t  \rightarrow 0$.    \end{proof}

\end{document}